\documentclass[11pt,a4paper]{article}
\usepackage{amsmath,amssymb,amsthm,datetime,enumerate,extarrows,fullpage,mathrsfs,mathtools,parskip,tikz}
\usepackage[all,cmtip]{xy}
\usetikzlibrary{arrows,decorations.pathreplacing}

\usepackage{graphics}

\numberwithin{equation}{section}
\numberwithin{figure}{section}

\begingroup 
\makeatletter 
\@for\theoremstyle:=definition,remark,plain\do{
\expandafter\g@addto@macro\csname th@\theoremstyle\endcsname{
\addtolength\thm@preskip\parskip 
}
}
\makeatother
\endgroup

\newtheorem{theorem}{Theorem}[section]
\newtheorem{lemma}[theorem]{Lemma}

\newtheorem{corollary}[theorem]{Corollary}
\theoremstyle{remark}
\newtheorem{remark}[theorem]{Remark}
\theoremstyle{definition}

\newcommand*{\defeq}{\mathrel{\vcenter{\baselineskip0.5ex \lineskiplimit0pt\hbox{\scriptsize.}\hbox{\scriptsize.}}}=}

\newcommand{\GL}{\ensuremath{\operatorname{GL}}}

\newcommand{\emb}{\ensuremath{\operatorname{Emb}}}
\newcommand{\diff}{\ensuremath{\operatorname{Diff}}}

\newcommand{\SO}{\operatorname{SO}}
\newcommand{\hocolim}{\operatorname{hocolim}}

\newcommand{\Fr}{\operatorname{Fr}}

\newcommand{\oD}{\text{int} (D)}
\newcommand{\oW} {\text {int} (W)}
\newcommand{\Iso}{\ensuremath{\operatorname{Iso}}}
\newcommand{\rel}{\ensuremath{\operatorname{rel}}}

\renewcommand{\mathbf}[1]{\ensuremath{\boldsymbol{#1}}}

\title{Homology stability for symmetric diffeomorphism and mapping class groups}
\author{Ulrike Tillmann}
\date{\monthname\ \the\year}

\begin{document}
	\maketitle



\begin{abstract}
For any smooth compact manifold $W$ of dimension of at least two 
we prove that the classifying spaces of its group of  
diffeomorphisms which
fix a set of $k$ points or $k$ 
embedded disks (up to permutation) satisfy homology stability. 
The same is true for so-called symmetric diffeomorphisms of $W$ 
connected sum with $k$ copies of an arbitrary compact smooth manifold $Q$ of the same dimension.
The analogues for mapping class groups as well as other generalisations 
will also be proved. 
\end{abstract}

	\section{Introduction} 
	\label{sec:introduction}

Compact smooth manifolds and their 
diffeomorphism groups are fundamental objects in geometry
and topology. If one wants to understand diffeomorphic
families of a given compact manifold $W$
one is led to study the associated topological moduli space
$$
\mathcal M (W) := \lim _{N \to \infty} \emb (W ; \mathbb R ^N)/ \diff (W), 
$$
the orbit space of the group of diffeomorphisms of $W$ acting on the 
space of smooth embeddings of $W$ into larger and larger Euclidean space.
This is also a
model for the classifying space of $\diff (W)$
and hence its cohomology is the set of characteristic classes for smooth
$W$-bundles. To understand the topology of $\mathcal M(W)\simeq B \diff (W)$ 
is a difficult problem and only  
a few special cases are fully understood. 

\vskip .1in
{\it Stabilisations:}

To simplify the problem one might consider to study an associated stabilised problem.   
One way to stabilise that has played an important role in the past is to  thicken
the manifold by taking the cross product with the unit interval $I = [0, 1]$ and to   consider the  pseudo-isotopies 
$P(W)$,  which are diffeomorphisms of $W \times I$ 
which fix $W \times \{ 0\}$ and $\partial W \times I$ point-wise. 
Repeating the process, the stable pseudo-isotopy group is then defined as the limit space 
$\mathcal P(W) := \lim _{k \to \infty} P( W\times I^k)$. It can be 
studied via $K$-theory. Indeed, in the late 1970s  
Waldhausen \cite {Waldhausen} proved that his $K$-theory 
$A(W)$ is a double deloop of $\mathcal P(W)$ times the well-understood
free infinite loop space on $W$.
That is to say, he proved that
$$
A(W) \simeq Wh (W) \times \Omega ^\infty \Sigma ^\infty (W_+)
$$ 
where $Wh(W) $ is the smooth Whitehead space of $W$ and 
$\Omega^2 Wh (W) \simeq \mathcal P(W)$.
The crucial fact that allows one to deduce in principle 
something for $P(M\times I^k)$ from the $K$-theory $A(W)$ is 
Igusa's stability theorem \cite {Igusa}.  
It says that the maps $P (M\times I^k) \to P(M\times
I^{k+1})$ are $c$-connected for $\dim W + k \geq \text{max} ( 2c+7, 3c+4)$. 
Though Waldhausen $K$-theory is well studied, 
it remains however difficult to deduce concrete 
information for specific manifolds.

\vskip .1in
More recently, another stabilisation process has been considered and its study
has proved very fruitful. Rather than increasing the dimension of the 
manifold we increase its complexity in the following sense. 
Let $Q$ be a manifold of the same dimension as $W$
and consider the connected sum 
$W\sharp_k Q$ 
of $W$ with $k$ copies of $Q$. 
So far the most important example is when $W$ is the sphere 
and $Q$ is the two-dimensional torus in which case $W\sharp_k Q$ is a surface of
genus $k$. As $k$ goes to infinity,
  the cohomology of the classifying space of its
diffeomorphism group is understood by the   Madsen-Weiss theorem 
\cite {MW}, see also \cite {GMTW}.  
In this case it is Harer's homology stability theorem for  the mapping 
class group of surfaces that allows us to 
deduce information for a compact surface.  
These theorems have recently been generalised  by Galatius and Randal-Williams
\cite {GRW1}, \cite {GRW2} to manifolds $W$ of even dimensions
$2d$ for $d>2$ and  with $Q = S^d \times S^d$, and by Perlmutter \cite{Perlmutter} in cases which include also some odd-dimension manifolds.

\vskip .1in
{\it Homology stability:}

Motivated by the second approach to stabilisation, we look at the question of homology 
stability for diffeomorphisms groups of manifolds more generally. 
Until the recent paper \cite {GRW2}, 
homology stability for (classifying spaces of) 
diffeomorphism groups was only known 
in the case of surfaces, and here only through the 
homology stability of the associated discrete groups, 
the  mapping class groups --
the mapping class groups  are homotopy equivalent to 
the diffeomorphism groups for surfaces  of negative Euler characteristic.
The purpose of this paper is to show that homology stability holds quite broadly 
for certain diffeomorphisms groups 
involving arbitrary  manifolds $W$ and $Q$ of both odd and even dimensions. 
The point of view taken is that the homology stability of the symmetric groups, 
and more generally of configuration spaces, is fundamental. The stability theorems 
we prove here are
derived from this basic example.  

\vskip .1in
{\it Main results:}

Throughout this paper,
let $W$ be a smooth, compact, path-connected manifold of dimension $ d\geq 2$ 
with a (parametrised) closed $(d-1)$-dimensional disk 
$\partial _0$ in its boundary $ \partial W$. Let 
$$ 
\diff (W) := \diff (W; \rel \partial_0)
$$
denote  its group of smooth diffeomorphisms that fix $\partial_0$ point-wise. 
More precisely, we will assume that any diffeomorphism $\psi$ fixes a 
collar of a  small neighbourhood of $\partial _0$ in $\partial W$ and furthermore
that restricted to a fixed collar of $\partial W$  
it is of the form $1 \times \psi |_{\partial W} $.
{\bf Note}, as they fix 
$\partial _0$ point-wise, 
the diffeomorphisms will automatically  be orientation preserving if $W$ is
oriented!

Let 
$W^k$ denote the manifold $W$ with $k$ distinct marked points $w_1, \dots, w_k$ in its interior.
Consider the subgroup  
$\diff^k (W)$ of $\diff (W)$ consisting of those diffeomorphisms that permute the marked points.

\begin{theorem}
There are maps $H_* (B\diff^k (W) )
\to H_* (B\diff ^{k+1} (W))$ which are split injections 
for all degrees $*$ and  isomorphisms in degrees $* \leq k/2$.
\end{theorem}

Now fix 
$k$ closed disjoint disks in $W \setminus \partial W$ and 
parametrisations $\phi_1, \dots , \phi _k : D^k \to \partial W \setminus \partial W$ which are compatible with the orientation of $W$ if it is oriented.
Then consider the group $\diff_k (W)$ of  diffeomorphisms $\psi \in \diff (W)$ which
commute with these parametrisations in the 
sense that 
for some permutation $\sigma \in \Sigma _k$ depending on $\psi$ we have  
$$
\psi \circ \phi_i = \phi_{\sigma (i)} \quad \text{ for all } i = 1, \dots , k. 
$$

\begin{theorem}
There are maps $H_* (B\diff_k (W) )
\to H_* (B\diff _{k+1}(W))$ which are split injections for all degrees $*$ 
and  isomorphisms in degrees $* \leq k/2$.
\end{theorem}

Denote by $W_k$ the manifold $W$ with the interior of the $k$ parametrised disks removed.
Let $Q$ be another smooth, connected manifold of dimension $d$. 
Remove an open disk and 
glue $k$ copies of it to $W _k$ to form the connected sum $W\sharp_k Q$.
We will be interested in those diffeomorphisms that map $W_k$  (and hence the disjoint union of the $k$ 
copies of $Q_1$) onto themselves. More precisely, let $G \subset O(d) 
\subset \diff (S^{d-1})$ be a closed 
subgroup. When $W$ is orientable we assume that  $G \subset SO(d)$.
Let $\Sigma_G \diff (W\sharp_k Q)$ denote the diffeomorphisms of $W\sharp_k Q$
that map  $W _ k$ onto itself and restrict to an element of the wreath product 
$\Sigma _k \wr G$ on the $k$ 
boundary spheres. 

\begin{theorem}
There are maps $H_* (B\Sigma_G \diff (W \sharp _k Q) )
\to H_* (B\Sigma_G \diff (W\sharp _{k+1} Q))$ which are  split injections 
for all degrees $*$ and isomorphisms in degrees $* \leq k/2$.
\end{theorem}

We will prove a more general result, Theorem 3.4, which may be thought of as a homology stability theorem for diffeomorphisms of $W_k$ with decorations given by some subgroup of the diffeomorphisms of $Q$ or indeed any other group. 
For all three theorems we will also prove the stronger statement that 
the maps of the underlying spaces  are 
stable  split injections (in the sense of 
stable homotopy theory).
This uses the results in \cite {MT}. 
The other basic ingredients are a generalisation 
of what is known in dimension two as the Birman 
exact sequence  and repeated spectral sequence arguments 
building on the homology stability of configuration 
spaces with labels in a fibre bundle.
After collecting some preliminary results in section 2, the 
three theorems above are proved sequentially in section 3. 

There are also analogues of these results for mapping class groups, 
see Theorems 5.1, 5.2, and 5.3.
Theorem 5.1 had previously been proved 
by Hatcher and Wahl \cite {HatcherWahl}.
Theorem 5.2 seems to be new even in the case of surfaces. 
After establishing the homology stability of the fundamental group of
configuration spaces with twisted labels in section 4, the mapping class 
group analogues of the above theorems are stated and  
proved in section 5.


In section 6 we collect some variations, generalisations and consequences 
of our results.


\section{Preliminary results}

In this section we will  recall  the homology stability and the stable splitting theorems for 
configuration spaces  with
twisted labels. 
We will also recall some basic results relating to the diffeomorphism groups and describe 
the spectral sequence argument that we will use repeatedly.

\subsection{Homology stability and stable splittings for configuration spaces}

Let $W$ be a smooth, compact, path-connected  manifold with non-empty boundary $\partial W$ of dimension $d\geq 2$. We fix a collar of $\partial W$ and a $(d-1)$-dimensional disk 
$ \partial _0$ in $\partial W$. Furthermore,   
let $\pi:E\rightarrow W$  
be a fibre bundle and assume that for each $w\in W$, the fibre 
$E_w$ over $w$
is path-connected. 
We define the configuration space of $k$ ordered particles in $W$ with 
twisted labels in $\pi$ as
\begin{equation*}
\widetilde C_k(W;\pi)\defeq \{(\mathbf{w},\mathbf{x})\in {\oW }^k \times E^k :
\pi (x_i) = w_i , w_i \neq w_j \text { for } i\neq j , 1 \leq i,j, \leq k \}.
\end{equation*}
and denote by $C_k (W; \pi)$ its orbit space under the natural action by the symmetric group $\Sigma _k$ by permutations.
Here $\oW$ denotes the interior of $W$.
When $E= W\times X$ for some space $X$ and $\pi$ is the projection onto $W$, this is the more familiar 
space of configurations of $W$ with labels in $X$ which we also denote by $C_k(W; X)$. When
$X$ is a point we use the notation $C_k(W)$.
The twisted labels we will  primarily be interested in 
are the 
tangent bundle $\tau: TW\to W$  and its associated frame bundle $\Fr$. 

We  now define inclusion maps that allow us to think of the
configurations of $k$ points as a subspace of the configuration
space of $k+1$ points.
Define $W^+ := W\cup \partial W\times [0,1]$ to be $W$ with a collar glued to its boundary, and extend $\pi $
to a bundle $\pi^+: E^+ \to W^+$ by
pull back along the natural projection $W^+ \to W$.
Fix points $w_0 \in \text {int} (\partial_0 \times [0, 1])$ and  $x_0
\in E^+_{w_0} $.
We define a stabilisation map
\begin{equation}
\tag{2.1}
 C_k (W; \pi) \overset \sigma \longrightarrow C_{k+1} (W^+; \pi^+)
        \longrightarrow C_{k+1} (W; \pi)
\end{equation}
where the first map adds the particle $(w_0,x_0)$ to any configuration
and the second map is a homeomorphism
induced by a map that isotopes $W^+$ into $W$ mapping
$\partial W\times [0, 1]$ into a collar of $\partial W$ in $W$ while
at the same time leaving the complement of the collar in $W$ fixed. We will generally suppress this map and identify $W^+$ with $W$.

Let $\diff (W; \pi)$ denote the group of automorphims of $\pi$ that cover a 
diffeomorphism $\beta \in \diff (W)$ with $\beta^* (\pi)$ and $\pi$ 
isomorphic as bundles over $W$. When $\pi$ is a trivial bundle or the tangent bundle
of $W$ then $\diff(W; \pi)$ contains $\diff(W)$.
Note that because of our assumption on the restriction of a diffeomorphism to a collar of $\partial W$ there is a natural extension map
$\diff( W; \pi) \to \diff (W^+; \pi^+)$ mapping $\psi$ to $\psi^+$ which is
defined on the extension by
$\psi ^+|_{\partial W \times [0,1]} =  \psi_{\partial W}\times 1$. In particular, $\psi^+$  fixes  $\pi^+|_{\partial_0 \times [0, 1]}$. 
With these definitions
it is easy to check that the inclusion $\sigma$ is $\diff(W; \pi)$-equivariant.

We will need the following two theorems. 

\begin{theorem} 
The stabilisation map $\sigma: C_k (W;\pi) \to C_{k+1} (W; \pi)$ is 
$\diff(W; \pi)$-equivariantly stably split injective.
\qed
\end{theorem}

This is well-known in the non-equivariant setting by the Snaith splitting 
theorem when $W$ is Euclidean space and by a theorem of B\"odigheimer \cite 
{Boedigheimer} for general $W$. That these splittings are 
$\diff(W; \pi)$-equivariant was proved in \cite{MT}.

The following is a  generalisation of the standard homology stability for configurations 
spaces in a manifold, see  \cite {Segal}, \cite {McDuff}, \cite {RW}, to configuration spaces with 
 labels in a bundle. Indeed, the proof in \cite {RW} can be generalised to this case.  
For details of this proof and alternative proofs  
we refer to \cite {CPalmer} and  \cite {KupersMiller}.

\begin{theorem}
The stabilisation map $\sigma: C_k (W;\pi) \to C_{k+1} (W; \pi)$ induces an isomorphism
on  homology in degrees $* \leq k/2$. \qed
\end{theorem}

\subsection{ Some lemmata concerning diffeomorphism groups}

We will  need the following well-known transitivity properties of the diffeomorphism groups.

\begin{lemma} 
(i) Given any two sets $\{x_1, \dots , x_k\}$ and $\{w_1, \dots, w_k\}$
 of $k$ points in the interior of the connected
manifold $W$ there is a diffeomorphism $\psi \in \diff (W; \partial_0)$ such 
that $\psi (x_i) = w_i$ for all $i= 1, \dots, k$.

(ii) Furthermore, when $W$ is not orientable,
given any linear maps $L_i: T_{x_i}W \to T_{w_i} W$,  
$\psi $ may be chosen so that $D_{x_i}\psi = L_i$. When $W$ is orientable
the same is true as long as the $L_i$ are all orientation preserving.
\end{lemma}

\begin{proof}
For (i) choose paths connecting $x_i$ to $w_i$ and consider the associated 
slides. As $d\geq 2$, the paths and slides 
can be chosen so as to avoid any of the other points $x_j$ 
and $w_j$ for $j\neq i$. Performing one slide after the other defines a 
diffeomorphism $\psi$ with the required properties. 

For (ii) first note that by part (i) we may assume that $x_i = w_i$. Next note that if $W$ is non-orientable 
we can choose a path connecting $x_i$ to itself
along which the slide induces an orientation reversing map of tangent spaces. 
This corresponds to an element in $\pi_1 (W; x_i)$ mapping to the  non-trivial
element in $\mathbb Z/ 2\mathbb Z = \pi _1 (BGL_d(\mathbb R))$ via the map induced by the classifying map of the tangent bundle.
Thus, it is enough to construct for any orientation preserving linear map $L_i$ from $ T_{x_i}W  $
to itself, a diffeomorphism $\alpha_i$ with $D_{x_i} \alpha_i = L_i$
that is the identity
outside a small neighbourhood of $x_i$. 
\footnote{The existence of such maps can also be deduced 
from  (2.2) below with $W=M$ and $K$ the union of $k$ closed disks 
centered at $ x_1, \dots, x_k$.}

To do so, choose a metric on $W$ 
and $\epsilon >0$ small enough so that the exponential map $\exp : T_{x_i} W \to W$
is injective on the $\epsilon$-ball around zero and the closure of its image does not contain $x_j$ for $i\neq j$.
Define $\alpha_i :W \to W$ to be the identity outside this image and otherwise by
$$
\alpha _i (x) := (\exp \circ \Phi \circ \exp ^{-1}) (x)
$$
for some diffeomorphisms $\Phi$  of $T_{x_i}W$ satisfying $D_0\Phi = L_i$ and $\Phi (v) =v$ for $|v| \geq \epsilon$.
Then in particular $D_{x_i} \alpha _i = L_i$.
It remains to prove $\Phi$ exists, 
and we  now sketch the construction of such a $\Phi$.
 
As $L_i$ is orientable, we may choose a smooth path $A(t)$ in $\GL^+ _d (\mathbb R)$ from $L_i$ to  the identity
which is defined for $t \in [\epsilon/2, \epsilon]$ and constant near the boundary. 
For $|v| \in [\epsilon /2, \epsilon]$  define 
$$
\Phi (v) := A(|v|) (v) (|v|/|A(|v|)(v)|)  . 
$$
The multiplication by the scalar $ |v|/|A(|v|)(v)|$ ensures that 
the sphere of radius $|v|$ is mapped to itself.  Thus $\Phi$ is 
a diffeomorphism on
its domain of definition and we can extend it by the identity for $v$ 
with $|v|
\geq \epsilon$.
For $|v| \leq \epsilon /2$ define
$$
\Phi (v) = (\rho \circ  L_i)( v)
$$
where $\rho $ is a suitable diffeomorphisms of $T_{x_i} W$ that maps the convex image under $L_i$ of the $\epsilon /2$-ball back to itself, 
for example
by using a radial contraction/expansion map which is constant around $0$. Then $\Phi$ is also a diffeomorphism on the $\epsilon /2$-ball
and satisfies $D_0\Phi = L_i$.
\end{proof}

The group $\diff_1 (W)$ can be identified with the group
$\diff ( W; D \coprod \partial_0)$
of diffeomorphisms of $W$ that fix an embedded closed disk $D$ point-wise in addition to
$\partial_0$,
and $\diff ^1(W) $ can be identified with the group
$\diff (W; w_1 \cup \partial _0)$ of diffeomorphisms that fix the additional point $w_1$.
Let $\diff (W; w_1 \cup \partial _0; T_{w_1}W)$
denote the group of diffeomorphisms
which fix the point $w_1$ and its tangent space.
With these identifications there are canonical injective homomorphisms
$$
\diff_1 (W)  \overset \iota \longrightarrow 
\diff (W; w_1 \cup \partial_0; T_{w_1}W) \longrightarrow
\diff^1 (W)
$$
connecting these three groups, each forgetting some of the structure.

\begin{lemma} The map
$\diff_1 (W)  \overset \iota \longrightarrow
\diff (W; w_1 \cup \partial_0; T_{w_1}W)$
is a homotopy equivalence.
\end{lemma}

In the proof we will repeatedly use the following result, see \cite {Palais}.
Let $K$ be a compact sub-manifold of $W$ and $M$ be another smooth manifold. 
Then
the restriction map from the space of smooth embeddings of $W$ into $M$
to the space of smooth embeddings of $K$ into $M$ is a locally trivial 
fibre bundle:
\begin{equation*}
\tag {2.2}
\text{Emb } (W, M) \longrightarrow \text {Emb } (K, M).
\end{equation*} 
When $W=M$  the space of embeddings $\text{Emb } (W,M)$ may be replaced by the space of 
(compactly supported) diffeomorphisms $\diff (W)$, see \cite {Lima}.

\begin{proof}
Using (2.2) with $K=D$, we can show that the following 
are maps of (horizontal) fibre bundles
\begin{equation*}
\xymatrix{
\diff _1( W) \ar[d] ^\simeq \ar[r] &\diff (W) \ar[d]^ = \ar[r] 
	&\text{Emb} (D, W) \ar[d]^\simeq \\
\diff ^1( W;w_1 \cup \partial_0; T_{w_1}W) \ar[d]\ar[r] &\diff (W) \ar[d]^ = \ar[r] 
	&\Fr \ar[d]\\
\diff ^1( W) \ar[r] &\diff (W) \ar[r] & W
}
\end{equation*}
where the right horizontal maps are the restriction maps to $D$, $Iso (
T_0 D,  T_{w_1} W)$
and $w_1$.
Furthermore, evaluation at $0 \in D$ and projection from the frame bundle to 
$W$ give rise to compatible fiber bundles
\begin{equation*}
\xymatrix{
\text {Emb} (D, 0; W, w_1) \ar[d]^\simeq \ar [r] 
	&\text{Emb} (D, W) \ar[d] ^\simeq \ar[r] &W \ar[d]^= \\
\GL_d (\mathbb R) \ar[r] &\Fr \ar[r] &W.
}
\end{equation*}
We will now argue that the left arrow is a
 homotopy equivalence, hence so is the middle one, 
and therefore also the two arrows in the previous
diagram which are labelled as homotopy equivalences.

The space of collars for the image of $D$ in $W$ is contractible for any embedding. Hence, 
\begin{equation*}
\tag  {2.3}
\text {Emb} (N(D), 0; W, w_1)\simeq 
\text {Emb} (D, 0; W, w_1)  
\end{equation*}
where $N(D)$ is an open disk containing $D$. The fiber over the identity of the forgetful map
\begin{equation*}
\tag {2.4}
\text {Emb} (N(D), 0; W, w_1) \longrightarrow \Fr|_{w_1} \simeq  \GL_d (\mathbb R)
\end{equation*}
is the space of tubular neighborhoods of $w_1$ in $W$ and hence is contractible.
(The intuition is that by combing 
from the origin  $w_0$  any diffeomorphism fixing its tangent space  can be homotoped
to the identity in some neighbourhood of $w_0$.)
\end{proof}

We will be interested in a variation of this lemma.
Fix $G\subset O(d)$ and denote  by $\diff_G(W_1)$  the group of diffeomorphisms of $W_1 = W\setminus \text{int}(D)$ that
restrict on (a neighbourhood of) the  boundary sphere to an
element in $G$.
There are canonical injective homomorphisms
$$
\diff_1 (W) \longrightarrow \diff _G (W_1) \overset \iota \longrightarrow \diff^1 (W)
$$
where the left map is the inclusion, which is defined for all $G$,  and the  right map is defined by extending
the diffeomorphisms canonically from the boundary sphere  $S^{d-1}$ to the disk $D^d$. This is where we use the condition that $G \subset O(d)$.
Note that the midpoint $w_1$ of $D^d$ is fixed and the induced map on its tangent space is the same 
element of $O(d)$ as that defining the map on $S^{d-1}$.  
Let $\diff^1_G (W )$ denote the subgroup of $\diff ^1(W)$ consisting of those diffeomorphisms  whose induced map on the tangent space $T_{w_1}W$
is an element in $G$.

\begin{lemma}
The map $\diff _G (W_1) \overset \iota \longrightarrow \diff^1_G (W)$
is a homotopy equivalence.
In particular, the map $\diff _{O(d)} (W_1) \overset \iota \longrightarrow \diff^1 (W)$ is a homotopy equivalence.
\end{lemma}

\begin{proof}
Consider the following commutative diagram of fibre bundles
\begin{equation*}
\xymatrix{
\diff_1 (W) \ar[d]^\simeq \ar[r] &\diff_G(W_1) \ar[d]^\iota \ar[r] &G \ar[d]^= \\
\diff(W; w_1 \cup \partial _0; T_{w_1}W) \ar[r] &\diff^1_G(W) \ar[r]&G,
}
\end{equation*}
where the vertical arrows are the canonical inclusion maps and the right 
horizontal arrows are 
given by restriction to the boundary and the tangent space 
at $w_1$ respectively.
Indeed, the top row is the part of the bundle
$$
\diff (W_1) \longrightarrow \diff (\partial W_1 \setminus \partial W)
$$
lying over the subgroup $G$. Similarly, the bottom row is the part of 
the bundle
$$
\diff ^1 (W) \longrightarrow \GL_d \mathbb R
$$
lying over the subgroup $G$.
 As the map on
the base spaces is the identity and the map on the fibre spaces 
is a homotopy equivalence by Lemma 2.4, it follows that also the map of 
total spaces is a homotopy equivalence. The second statement in the lemma 
follows from the homotopy equivalence $O(d) \simeq \GL_d (\mathbb R)$.
\end{proof}

We finish this section with some remarks that put our choices of groups above into context.
Taking  $W=M= S^d$ and $K= D^d$  in (2.2) we  get a fibre bundle
\begin{equation*}
\tag{2.5}
\diff (D^d; \, \text{rel }  \partial ) \longrightarrow \diff (S^d) \longrightarrow 
\text {Emb } (D^d, S^d) 
\end{equation*}
for all $d> 0$.
The orthogonal group $O(d+1)$ injects via $\diff (S^d)$
into $\text {Emb } (D^d, S^d) $. Following this by the evaluation 
map at the origin $0\in D^d$ gives rise to a map of fibre bundles
\begin{equation*}
\xymatrix{
O(d) \ar[d] ^\simeq \ar[r] &O(d+1) \ar[d]^\simeq \ar[r]  & S^d\ar[d]^= \\
\text {Emb } (D^d, 0; S^d, w_1) \ar[r] &\text {Emb } (D^d, S^d) \ar[r]& S^d
}
\end{equation*}
where all vertical maps are homotopy equivalences by (2.3) and (2.4).
We thus have a homotopy equivalence
of spaces
\begin{equation*}
\tag{2.6}
\diff (S^d) \simeq O(d+1 ) \times \diff (D ^d; \, \text{rel } \partial).
\end{equation*}
Note that here $\diff (S^d)$ includes the orientation reversing diffeomorphisms. Restricting to the orientation preserving
diffeomorphisms this is Proposition 4 of \cite{Cerf68}, page 127.

\begin{remark}
By theorems of Smale \cite{Smale} and Hatcher \cite {Hatcher} the group  $\diff (D^d, \text{rel } \partial)$
is contractible for $d=2, 3$.
For $d\geq 5$, the mapping class group $\pi_0 (\diff (D^d, \text{rel } \partial))$ is known
to be isomorphic to the group $\Theta _{d+1}$ of exotic spheres in dimension $d+1$ by theorems of
Smale \cite {Smale61} and Cerf \cite {Cerf70}. The exotic spheres are constructed by a clutching construction that glues two disks
$D^{d+1}$ together using elements in $\diff (D^d; \text{rel } \partial)
\subset \diff(S^d)$.
\end{remark}

\vskip .2in
The question arises in our context which diffeomorphism of the $k$ boundary
spheres is the  restriction of  some symmetric diffeomorphism of
$W\sharp_k Q$. First note, that if $\alpha \in \diff _0 (S^{d-1})$ is an
element in the identity component, then it is such a restriction as any path to the identity
defines a diffeomorphisms $\hat \alpha$ of $S^{d-1} \times [-1,1]$ with
$\hat \alpha |_{S^{d-1} \times \{ 0 \}} = \alpha$
and $\hat \alpha |_\partial = 1$
which can be extended to the whole of $W\sharp _k Q$ by the identity.
By the Lemma 2.3 then also any element in the wreath product
$\Sigma_k \wr \diff_0 (W\sharp_k Q)$ is in the image of the restriction map.
When $W$ is not orientable, this holds also if the 
$-1$ component defined by the non-zero element in $\pi_0 (GL_d(\mathbb R))$ is included.
Our methods here however force us to restrict ourselves to the  case when
$G \subset O(d)$ is a closed subgroup.

\subsection{ Spectral sequence arguments and stable splittings}

We will use the following well-known spectral sequence arguments repeatedly.

\begin{lemma}
Let  $ f: E^\bullet_{p,q} \to \widetilde E^\bullet _{p,q}$ be a map of homological first quadrant spectral sequences.
Assume that 
$$
f: E^2_{p,q} \overset \simeq \longrightarrow \widetilde E^2_{p,q}
\quad \quad \text { for } 0\leq p < \infty \text { and } 0\leq q \leq l.
$$
Then $f$ induces an isomorphism on the abutments in degrees $* \leq l$.   
\end{lemma}

\begin{proof}
The fact 
that $f$ is an  isomorphism on $E^2_{p,q}$ for all $p$  allows us to control higher differentials
with targets of bidegree $(p,q)$ with $p+q\leq l$. Also note that that $r$-th
differential has bi-degree $(-p, q-1)$. So the differentials from terms with bidegree $(p,q)$ and $p+q \leq l$ have targets with total degree no greater than $l$.
The details are left to the reader. 
\end{proof}

We can improve the situation when we have further constraints.
Consider the following abstract version of the 
situation described in Theorem 2.1.
For $k\geq 1$, let $G$ be a topological group and $C_k$ be 
$G$-spaces with $G$-equivariant 
maps $C_k \to C_{k+1}$ for $k \geq 0$ that are stably $G$-equivariantly 
split injective. 
Consider the  map induced on  Borel constructions
$$
f_k :EG\times_ G C_k \longrightarrow EG\times_G C_{k+1}.
$$
We have the following helpful tool. 

\begin{lemma}
If $C_k \to C_{k+1}$ is  stably $G$-equivariantly 
split injective, then so is 
$f_k $.
In particular $f_k$ is split injective in homology.
\end{lemma} 

\begin{proof}
The assumptions on the $C_k$ imply that it is $G$-equivariantly
stably homotopy equivalent to 
$V_k = \bigvee_{j=0} ^k D_j$ with $D_j = C_j/C_{j-1}$. 
The maps $C_k \to C_{k+1} $ correspond 
$G$-equivariantly to
the split injective maps $V_k \to V_{k+1}$. These  also induce split injective 
maps on the Borel constructions as
$$
EG\times_G V_k \simeq \bigvee _{j=0} ^k   EG \ltimes _G D_j
$$
where $X\ltimes Y $ denotes the half-smash $X\times Y/ X \times *$.
\end{proof}



\section{Proofs of the main theorems}

We will first define explicit stabilisation maps and then prove  the theorems in sequence using the results of the previous section.

\subsection {Definition of stabilisation maps }

We will construct maps of the diffeomorphisms groups that will induce the 
maps claimed in the three theorems of the introduction. They correspond on the underlying manifolds to connected sum but more precisely will be 
constructed by extending diffeomorphisms from $W$ to $W^+$ as encountered in section 2.1.

As before, let $\partial _0$ be an embedded $(d-1)$-dimensional disk in the boundary $\partial W$ and $Q$ be another $d$-dimensional manifold. 
Let $D$ be a parametrised closed disk in $Q$ and $D_0$ be a closed disk with centre $z_0$ contained in 
$\partial _0 \times [0,1] \subset \partial W \times [0,1]$. Remove the two disks and  
identify their boundary to form a new manifold $Q'$. 
To form the connected sum we take the union of $W$ and $Q'$ and identify $\partial W$ in $W$ with 
$\partial W \times \{0 \}$. Note that the resulting manifold $W\sharp Q$ is indeed
diffeomorphic to the usual connected sum of $W$ and $Q$.  It comes equipped with the embedded disk 
$\partial _0^+ = \partial_0 \times \{ 1 \}$ in its boundary, and we may thus repeat the process to 
form the connected sum $W\sharp _k Q$ as illustrated in Figure 3.1. 
 
\vskip .3in
\begin{center}
\includegraphics [height =1.5in] {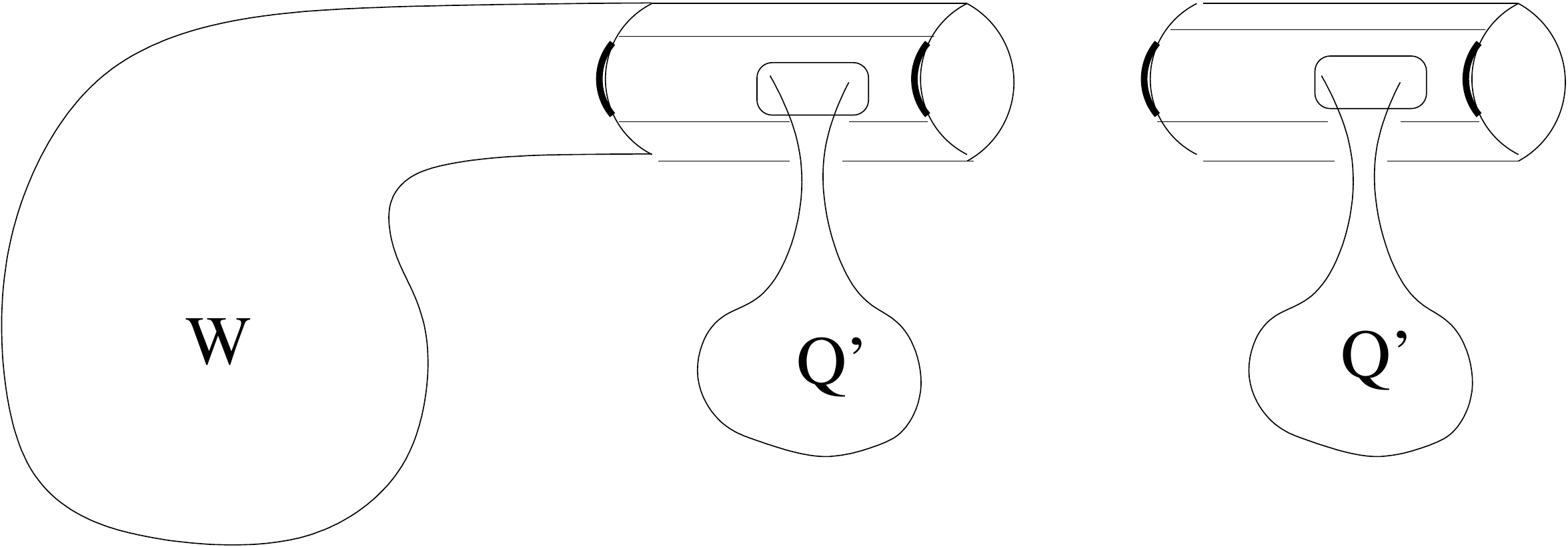}
\end{center}

\hfil {\it Figure 3.1: 
The connected sum of $W$ with  copies of $Q$.}
\hfil 

The $k$-fold connected sum
$W\sharp _kQ$ contains $k$ copies of $Q_1 = Q\setminus \oD$. 
The group $\Sigma _G\diff (W\sharp _kQ)$ of the introduction consists of the elements in the group $\diff(W\sharp _kQ; \coprod_k Q_1; \partial_0)$ of 
diffeomorphisms that fix $\partial_0$ point-wise, map $\coprod_k Q_1$ to itself and restrict on each boundary component to an element in $G \subset
O(d)$.  
Let $G \subset  \SO (d)$ when $W$ is orientable, or  
$G \subset O(d)$ when not.
When $G$ is the trivial group $e$, this group is simply the wreath product 
$$
\diff _k (W) \wr \diff_1 (Q) = \diff_k(W) \ltimes (\diff_1(Q) )^k.
$$
The diffeomorphisms  of $W\sharp _kQ$ that fix $\partial _0$ point-wise can 
be extended to diffeomorphisms of $W\sharp_{k+1} Q$ using the identity on 
the additional  copy of $Q_1$. We have thus well-defined
homomorphisms 
$
\sigma: \diff (W\sharp_k Q) \to \diff (W\sharp _{k+1} Q)
$
that restrict to homomorphisms
$$
\sigma: \Sigma_G \diff (W\sharp_kQ) 
\longrightarrow \Sigma _G\diff (W\sharp_{k+1} Q).
$$

Note that when $Q_1 = D_0 \setminus z_0$ is a disk with a puncture,
$W\sharp _k Q \simeq W^k$, and  when $Q_1 = D_0 \setminus \oD $ is a disk without its interior, 
$W\sharp _k Q \simeq W _k$.
By abuse of notation we will suppress this diffeomorphism and write
$$
\sigma: \diff^k (W) \longrightarrow \diff^{k+1} (W) \quad
\text { and } \quad \sigma: \diff _k (W) \longrightarrow \diff_{k+1} (W)
$$ 
for the resulting stabilisation maps.
The iterated stabilisation maps $\sigma ^k : \diff (W) \to \diff^k (W)$ and $\sigma ^k : \diff (W) \to \diff_k (W)$
have homotopy left inverses.
To define these, identify $\diff ^k(W)$ and $\diff _k(W)$ with diffeomorphisms that fix a set of $k$ points respectively of $k$ parametrised disks. 
The left inverse is given by the forgetful maps:
\begin{equation*}
\tag{3.1}
\diff^k(W) \longrightarrow \diff(W) \quad \text{ and } \quad  \diff_k(W) \longrightarrow \diff (W).
\end{equation*}

We note here that by Remark 2.6,  Theorem 1.1 and Theorem 1.2 are {\it not } special 
instances of Theorem 1.3.
In particular, $\Sigma_e \diff (W\sharp_k D^d)$ is 
homotopy equivalent 
to $\diff_k (W)$ when $d=2$ or $3$ but not in general. 
Nevertheless, all three are special cases of Theorem 3.4 below.

We will prove the three theorems of the introduction  in sequential order in the next sections.


\subsection {Proof of Theorem 1.1}

We will prove the following slightly stronger statement in which the split injection of homology groups is promoted to a stable (in the sense of stable homotopy theory) split injection of spaces.

\begin{theorem}
The stabilisation map $\sigma: B\diff^k(W) \to B\diff ^{k+1} (W) $ is stably split injective and a homology isomorphism
in degrees $* \leq k/2$.
\end{theorem}

\begin{proof}
Fix $k$ points $x_1, \dots , x_k$ in the interior of $W$ and consider the
evaluation map
$$
E_k: \diff (W) \longrightarrow C_k (W)
$$
which maps a diffeomorphism $\psi$ to $\psi (x_1), \dots, \psi (x_k)$.
Note  that  this map is surjective as by Lemma 2.4 (i) the diffeomorphism group of $W$ is $k$-transitive. 
It is a fibre bundle with fibre $\diff^k (W)$ by \cite {Palais}; see (2.2).
Extending this fibration twice to the right we get a new fibration
\begin{equation}
\tag {3.2}
C_k (W) \longrightarrow B\diff^k(W) \longrightarrow B\diff (W).
\end{equation}
This fibration can be constructed geometrically as follows.   
Fix an embedding $\partial_0 \to \{ 0\} \times \mathbb R^\infty$ and 
consider the space of smooth embeddings 
$$
E\diff (W; ) := \emb ^{\partial_0}( W ;
(\infty , 0] \times \mathbb R^\infty)
$$ 
extending it and which furthermore are of the form $1 \times h|_{\partial W}$ near 
$(-\epsilon, 0] \times \mathbb  R^\infty$. This is a contractible 
space by the Whitney embedding theorem and the diffeomorphism group
$\diff (W ) $ acts freely on it. The orbit space, 
the space of embedded manifolds $W'$ diffeomorphic to $W$, is a model for 
$B\diff (W)$. Similarly, the orbit space for the
restricted action of $\diff ^k(W)$, 
the space of embedded manifolds $W'$ with $k$ marked points,  gives a model for $B\diff^k (W)$. With 
these models the right hand map in the above fibration (3.2) is the map that 
forgets the marked points on the embedded manifold  and 
the fibre over  $W'$ is its $k$-fold configuration space. 

The fibration is compatible with stabilisation in the sense that 
we have maps of fibrations
\begin{equation*}
\xymatrix{
C_k (W) \ar[r] \ar[d]^\sigma & B\diff^k(W) \ar[r] \ar[d]^\sigma & B\diff (W)	\ar[d]^\sigma\\
C_{k+1} (W) \ar[r] &B\diff^{k+1}(W) \ar[r] & B\diff (W)
}
\end{equation*}
with a homotopy equivalence of the base spaces.
This gives rise to a map of spectral sequences
\begin{equation*}
\xymatrix{
E^2_{**} = H_*( B\diff (W); H_* ( C_k(W))) \ar[d]^\sigma &\Longrightarrow  
& H_*(B\diff^k (W))\ar[d]^\sigma \\
E^2_{**} = H_*( B\diff (W); H_* ( C_{k+1}(W))) &\Longrightarrow & H_*(B\diff^{k+1} (W)) \\
}
\end{equation*}
By Theorem 2.2, the map of $E^2_{pq}$-terms is an isomorphism
for $q\leq k/2$ and all $p$. Hence, by Lemma 2.7, it is also an isomorphism of the total spaces in 
degrees $* \leq k/2$ and the homological statement Theorem 1.1 follows.

Our geometric model for the fibration (3.2) identifies 
$$
B\diff ^k (W) \simeq E\diff (W) \times _{\diff (W)} 
C_k (W).
$$ 
Thus  Theorem 2.1 and Lemma 2.8 imply furthermore that the 
stabilisation map $\sigma$ is stably split injective.
\end{proof}

\subsection {Proof of Theorem 1.2}

We will prove the following slightly stronger statement.

\begin{theorem}
The stabilisation map $\sigma: B\diff_k(W) \to B\diff _{k+1} (W) $ is stably split injective and a homology isomorphism
in degrees $* \leq k/2$.
\end{theorem}

\begin{proof}
As in the proof of Theorem 3.1, we fix $k$ points $ \bold x = (x_1, \dots , x_k ) $ 
in the interior of $W$. We modify the evaluation map to also 
record  the map of tangent spaces at these points. Thus a diffeomorphism $\psi$ is mapped to 
$(\psi (x_1), \dots, \psi(x_k); D_{x_1}\psi, \dots , D_{x_k} \psi)$ with each $D_{x_i}\psi \in 
\Iso (T_{x_i}W, T_{\psi(x_i)}W)$. After fixing an isomorphism $T_{x_i} W \simeq \mathbb R^d$ for each $i$
this defines a map 
$$
E_k^T: \diff (W) \longrightarrow C_k (W; \Fr),
$$
to the twisted configuration space 
where $\Fr : \Iso ( \mathbb R^d, -) \to W$ is the frame bundle of the tangent  bundle of $W$ with 
fibre $\Iso ( \mathbb R^d; T_w W)$ over $w$.  
It is understood here that in the case when $W$ is oriented, 
the linear isomorphisms preserve the orientation and the fibre is isomorphic to the group of orientation preserving invertible matrices
$\GL_d^+ (\mathbb R)$. Otherwise it is isomorphic to the whole linear group $\GL_d(\mathbb R)$.
By Lemma 2.3.(ii) the evaluation map  $E^T_k$ is surjective and its   
fibre  over $\bold x$ and the identity linear transformations 
is
$$
\diff ( W; \{ x_1, \dots , x_k , T_{x_1} W , \dots , T_{x_k} W \} ;  \partial _0 ),
$$
the group of diffeomorphisms that fix $\partial _0$ point-wise, 
fix the points  $\{ x_1, \dots , x_k \}$ as a set and their
tangent spaces point-wise up to permutation. 
By an obvious extension of Lemma 2.4, this group is homotopic to 
$\diff _k (W)$. 

The remainder of the proof follows now the pattern of the proof of Theorem 1.1.
We consider the associated 
fibre bundle
\begin{equation}
\tag {3.3}
C_k (W; \Fr) \longrightarrow B\diff_k( W) \longrightarrow B\diff (W).
\end{equation}
The stabilisation maps commute with all maps in the fibre sequence 
and hence induce a map of 
associated Serre spectral sequences with $E^2$-term 
$$
E^2_{**} = H_* ( B\diff (W); H_* (C_k (W;\Fr))).
$$ 
By Theorem 2.2 with twisted labels  in the frame bundle $\Fr$, 
the stabilisation map induces an isomorphism on  homology groups 
$H_* (C_k (W; \Fr))$ in degrees
$* \leq k/2$, thus on the $E^2$-term of the Serre spectral sequence and the abutment by Lemma 2.7, which proves the homological statement Theorem 1.2.

A geometric model for the total space is given by the space of 
embedded manifolds in infinite Euclidean space
with $k$ marked points and framings for their tangent spaces, that is
$$
B\diff_k(W) \simeq \emb^{\partial_0} (W; (\infty, 0] \times \mathbb R^\infty))
\times _{\diff (W)} C_k ( W; \Fr).
$$
Recall that  it is understood that if $W$ is oriented then
the framings are compatible with the orientation.
We note here that the action of the diffeomorphism group on the label 
space is non-trivial.
The full strength of Theorem 2.1  and Lemma 2.8 imply the stable splitting.
\end{proof}

\subsection {Proof of Theorem 1.3}

To prove Theorem 1.3 we will consider two fibration sequences. Recall from the
previous section  that 
the evaluation map gives rise to  a fibration (up to homotopy)
\begin{equation*}
\diff _k(W) \longrightarrow \diff (W) \overset
{E_k^T}  \longrightarrow  C_k (W; \Fr)   
.
\end{equation*}
The second fibration is
\begin{equation*}
P\diff _k(W) \longrightarrow  {P_G\diff}  (W_k) 
\longrightarrow G ^k
\end{equation*}
where the right map sends a diffeomorphism to its restriction to the boundary spheres; $P\diff_k(W) $ and $P_G\diff (W_k) $ denote the subgroups of $\diff_k(W) $ and $\Sigma_G\diff (W_k) $  that
do not permute the boundary spheres. 
Extending the fibrations twice to the right and dividing out by the (free) symmetric group action on base and total space, we can combine them into the following 
diagram where each column and row is a fibration (up to homotopy):
\begin{equation*}
\xymatrix{
G^k 	\ar[r] \ar[d]^= &C_k (W; \Fr) \ar[r] \ar[d] 	
	& C_k (W; \Fr /G) \ar[d]	\\
G^k	\ar[r] \ar[d] &B\diff _k (W) \ar[r] \ar[d]
	&B{\Sigma_G \diff}(W _k) \ar [d] \\
\ast	\ar[r]  &B\diff (W)  \ar[r] ^=  
	&B\diff(W).
}
\end{equation*}
Here the bundle $\Fr/G$ over $W$  is the fibrewise quotient space of 
the frame bundle $\Fr$
by the natural action of $G$. 
Note that the right column 
is associated to the evaluation fibration
$$
E^G_k: \diff(W) \longrightarrow C_k (W; \Fr /G)
$$
which is $E^T_k$ from  (3.3) composed with the quotient map.
As in the proof of Theorem 3.2, we can now apply Theorem 2.2 and Lemma 2.7 as well as Theorem 2.1 and Lemma 2.8
to the 
fibration of the right column to prove the following proposition which generalises 
both Theorems 3.1 and 3.2.

\begin{theorem}
For $G \subset O(d)$ a closed subgroup, the map
$\sigma: B{\Sigma_G \diff}(W _k)\to B\Sigma_G \diff(W_{k+1})$ is a stably
split injection and  induces an isomorphism in homology of degrees $* \leq k/2$.
\qed
\end{theorem}

We will now \lq decorate' the diffeomorphism group of $W_k$ at the embedded
parametrised disks with a label group $H$ which 
comes equipped with a continuous surjective homomorphism $\rho: H\to G$.  
Define
$\Sigma ^H_G \diff _k(W)$ to be the pullback under the induced map on wreath products and
the natural restriction map $\Sigma _G \diff _k (W) \to \Sigma _k \wr G$: 
\begin{equation*}
\xymatrix{
\Sigma ^H_G \diff _k(W) \ar[d] \ar[r] &\Sigma _G \diff _k (W) \ar[d]\\
\Sigma_k\wr H \ar[r] ^\rho  &\Sigma_k \wr G.
}
\end{equation*}
Write $H_1$ for the kernel of $\rho$.
Then there is a  fibration sequence
$$
BH_1 \longrightarrow O(d) \times_G BH_1 \longrightarrow O(d)/G
$$
where we take $BH_1 = EH/H_1$ so that it has a free $G$ action.
Combining this with the diagram above we have the following (up to homotopy)
commutative diagram with fibrations in each column and row
\begin{equation*}
\tag{3.4}
\xymatrix{
BH_1^k     \ar[r] \ar[d]^= &C_k (W; \Fr \times _G BH_1 ) \ar[r] \ar[d]     
        & C_k (W; \Fr /G) \ar[d]        \\
BH_1^k     \ar[r] \ar[d] &B\Sigma_G^H \diff_k (W ) \ar[r] \ar[d]
        &B{\Sigma_G \diff}(W _k) \ar [d] \\
\ast    \ar[r]  &B\diff (W)  \ar[r] ^=  
        &B\diff(W).
}
\end{equation*}
Thus we have a description of $B{\Sigma_G^H\diff}_k  (W ) $ 
as a fibre bundle over $B\diff (W)$ with fibre the configuration space 
$ C_k (W; \Fr \times_G BH_1)$. As in the proof of Theorem 3.2, an 
application of Theorem 2.2 and 
Lemma 2.7 and  
invoking the full strength of Theorem 2.1 and Lemma 2.8 gives the following most general statement.

\begin{theorem}
The stabilisation map $\sigma: B\Sigma_G^H \diff_k (W) 
\to B\Sigma_G ^H\diff_{k+1} (W) $ is a stably split injection and an isomorphism in homology of
degrees $*\leq k/2$.
\qed
\end{theorem}

\vskip .2in
It is now easy to see that our previous results are special cases 
of this.
Let $H= \diff _G (Q_1)$ be the group of diffeomorphisms of $Q_1$ that restrict to an element of $G$  
on the spherical boundary component created by deleting a disk from $Q$. Then 
$$
\Sigma _G ^H \diff _k(W) = \Sigma _G \diff (W \sharp_k Q).
$$
Thus Theorem 3.4 specialises to Theorem 1.3.
When $H=\{ e\} = G$ is the trivial group, then
$$
\Sigma _G ^H \diff _k (W) = \diff _k (W)
$$
and Theorem 3.4 gives Theorems 1.2 and 3.2.
To recover Theorems 1.1 and 3.1 put $H=G= O(n)$ respectively $H=G=SO(n)$
in the orientable case with $\rho$ the identity homomorphisms. 
Then, by Lemma 2.5, we have a homotopy equivalence
$$
\Sigma _G^H\diff_k (W) \simeq \diff^k (W).
$$


\section{Fundamental groups of configuration spaces }

In order to prove the mapping class group analogues of our main results  we will need the analogues of Theorem 2.1 and 2.2.
The role of the configuration spaces will be taken by  their fundamental groups which we will now compute first.

\begin{lemma}
Assume $\dim ( W) \geq 3$ and 
let $\pi: E \to W$ be a fibre bundle on $W$ with path-connected fibre $F$. Then
$$
\pi _1 (C_k( W; \pi) )= \Sigma_k \wr \pi_1 (E) .
$$
\end{lemma}

For $\pi = id: W\to W$ this is the braid group of $W$ and the result is well-known. For completeness we include here a proof of the 
more general case.

\begin{proof}
Consider the fibre bundle for the ordered configuration spaces
\begin{equation}
\tag{4.1}
E\setminus E|_{\bold w} \longrightarrow \widetilde C _{k+1} (W; \pi) \overset {eval_{1, \dots , k}} \longrightarrow \widetilde C_k (W; \pi)
\end{equation}
where the map $eval_{1, \dots ,k}$ remembers the  first $k$ points of the ordered configurations and forgets the last one. 
The fibre over the configuration $\bold x = \{ x_1, \dots , x_k \}$ with  $\pi (\bold x) = \bold w $ in $\widetilde C_k (W)$ is a point in $E$
which is not in any of the fibres $E|_{\bold w}$ over $\bold w$. 
The stabilisation map for ordered configuration spaces $\widetilde \sigma: \widetilde C_k (W; \pi) \to \widetilde C_{k+1}(W;\pi)$ 
defines a section (up to isotopy) of the above fibre bundle. Hence there is a split exact sequence of fundamental groups
\begin{equation}
\tag{4.2}
0 \longrightarrow
\pi_1( E\setminus E|_{\bold w}) \longrightarrow \pi_1 (\widetilde C _{k+1} (W; \pi)) \overset {eval_{1, \dots , k}} \longrightarrow 
\pi_1( \widetilde C_k (W; \pi)) \longrightarrow 0.
\end{equation}
The natural inclusion $incl: E\setminus E|_{\bold w} \hookrightarrow E$ factors as the composition of the inclusion of the fibre in (4.1) 
and the map
$eval_{k+1}: \widetilde C_{k+1} (W; \pi) \to E$ that maps an ordered configuration $\bold x = \{ x_1, \dots, x_k, x_{k+1}\} $ 
to its last point $x_{k+1}$.   

{\bf Claim.} {\it
The inclusion $incl: E\setminus E|_{\bold w} \hookrightarrow E $ induces an isomorphism on fundamental groups. 
}

{\it Proof of claim.} 
Note that $incl$ defines a map of fibre bundles covering the inclusion $ incl_W: W\setminus \bold w \hookrightarrow W$.
We thus have a map of exact sequences of homotopy groups
\begin{equation*}
\xymatrix{
\pi_2 (W\setminus \bold w) \ar[r] \ar[d]^{incl_W} &\pi_1 (F) \ar[r] \ar[d]^= &\pi_1 (E\setminus E|_{\bold w})
        \ar[r] \ar[d] &\pi_1 (W\setminus \bold w) \ar[r] \ar[d]^{incl_W} &\pi_0 (F) \ar[d]^=       \\
\pi_2 (W) \ar[r] &\pi_1 (F) \ar[r]  &\pi_1 (E) \ar[r] &\pi_1 (W) \ar[r] &\pi_0 (F).\\
}
\end{equation*}
As $\dim (W) \geq 3$, $incl_W$ induces an isomorphism on fundamental groups and
a surjection on second homotopy groups. An application of the Five Lemma finishes the proof of the claim.
\qed

Using  the claim we see that the group extension (4.2) is doubly split and thus trivial:   
$$
\pi_1 (\widetilde C _{k+1} (W; \pi) ) \simeq \pi_1 (E) \times \pi_1 (\widetilde C_k (W; \pi)).
$$
Hence, by induction, starting with $\widetilde C_1 (W;\pi) = \pi _1(E)$, we get  $\pi_1 (\widetilde C_k (W; \pi)) \simeq \pi_1 (E)^k$.

To deduce the result we want for  unordered configurations spaces consider the fibre bundle
$$
\Sigma _k \longrightarrow 
\widetilde C_k (W; \pi) \longrightarrow  C_k (W; \pi)
$$
and the associated short exact sequence of groups
\begin{equation}
\tag{4.3}
\pi_1 (\widetilde C_k (W; \pi)) \longrightarrow  \pi_1 (C_k (W; \pi)) \longrightarrow \Sigma_k.
\end{equation}
Let $D \subset W$ be an embedded disk of codimension zero. The restricted bundle $\pi|_D$ is trivialisable and hence in particular has a section. 
We can thus define    
inclusions $ C_k (D) \hookrightarrow C_k( D; \pi|_D) \hookrightarrow C_k (W; \pi)$.
As $\dim (W)\geq 3$, $\pi_1 ( C_k (D)) = \Sigma _k$ and the inclusion gives a splitting  of the  short exact sequence (4.3).
By definition $\Sigma_k$ acts by permuting the $k$ points in $E$ and we can deduce that 
the fundamental group of $C_k(W; \pi)$ is the wreath product.
\end{proof} 

Consider a general wreath product $\Sigma _k\wr G$ for a group $G$. 
Its classifying space has a configuration space model  given by
$$
B(\Sigma_k \wr G) \simeq E\Sigma_k \times_{\Sigma_k} (BG)^k \simeq C_k (\mathbb R^\infty; BG).
$$
The automorphism group $\text{Aut} (G)$ of $G$ acts on $BG$ by the naturallity of the 
bar construction. 
Applications of Theorem 2.1 and Theorem 2.2 give the following general result:

\begin{lemma} 
The stabilisation map
$
B(\Sigma_k \wr G ) \to B(\Sigma_{k+1} \wr G)
$
 is  (i) $\text{Aut} (G) $-equivariantly split injective and 
(ii) induces an isomorphisms in homology in degrees $* \leq k/2$.
\qed
\end{lemma}  

We note now  that the mapping class group $\Gamma (W;\pi) := \pi_0 (\diff (W;\pi)) $ 
acts on the fundamental group of $W$ and more generally on the fundamental groups of the
configuration spaces $C_k(W; \pi)$; here  we choose the base-point to be on $\partial _0$ so that it is fixed by the elements
in $\diff(W;\pi)$.
Applying  Lemma 4.2 with 
$G= \pi_1 (E)$ and $\Gamma (W; \pi)\subset \text{Aut} (G)$  and invoking Lemma 4.1
gives
the following desired result. 

\begin{corollary}
Assume $\dim ( W) \geq 3$ and
let $\pi: E \to W$ be a fibre bundle on $W$ with path-connected fibres.
The stabilisation map $\sigma : \pi_1 (C_k(W; \pi) )\to \pi _1 (C_{k+1} (W; \pi))$ induces
(i) a $\Gamma (W;\pi)$-equivariant stably split injective map on
classifying spaces; and
(ii) an isomorphism on homology in degrees $*\leq k/2$.
\qed
\end{corollary}


\section {Homology stability for mapping class groups}

We prove here the analogues of the homology stability of diffeomorphism groups  for the associated mapping class groups.
Theorem 5.1 was also proved by Hatcher and Wahl \cite {HatcherWahl} with different 
methods.  
When $W$ is a surface, which we assume has non-empty boundary, the mapping class groups are homotopic 
to the diffeomorphism groups, see \cite {EarleSchatz}. 
So the analogues for the mapping class groups 
follow immediately from the results for the diffeomorphism groups proved in section 3. 
(See also \cite {BoedigheimerT} for results closely related to Theorem 5.1 and 5.2.)
We therefore have to prove the statements below only in the case 
when $\dim (W) \geq 3$ and Corollary 4.3 applies.  

Let $\Gamma ^k(W):= \pi_0 (\diff^k (W))$ be the mapping class group of $W$ with $k$ punctures. 

\begin{theorem}
The  map $\sigma :B\Gamma ^k(W) \to B\Gamma ^{k+1} (W)$  is a stable split injection and an isomorphism in homology for degrees  $*\leq k/2$.
\end{theorem}


Define $\Gamma _k(W) := \pi_0 (\diff _k(W))$.

\begin{theorem}
The map $\sigma : B\Gamma _k(W) \to B\Gamma _{k+1} (W)$  is a stable split injection  and an isomorphism in homology for degrees $*\leq k/2$. 
\end{theorem}

%
%
\vskip .2in
Define 
$
\Sigma_G \Gamma (W\sharp_k Q) := \pi_0 (\Sigma_G \diff (W\sharp _k Q))
$ 
where 
$G\subset O(d) \subset \diff (S^{d-1})$ is a closed subgroup.
More generally, for any group $H$ and homomorphisms $\rho : H\to G$
define 
$$
\Sigma_G^H \Gamma_k (W):= \pi_0 ( \Sigma^H_G \diff_k (W)).
$$
As explained at the end of section 3, the following  generalises  Theorems 5.1 and 5.2.

\begin{theorem}
%
The map $\sigma : B\Sigma^H_G \Gamma _k(W ) \to B\Sigma^H_G \Gamma _{k+1} (W )$  
is a stable split injection  and an isomorphism in homology for degrees  $*\leq k/2$.

\end{theorem}

\begin{proof}
From the fibration of the middle column of diagram (3.4)
$$
C_k (W; \Fr \times _G BH_1) \longrightarrow B\Sigma ^H_G \diff_k (W) \longrightarrow 
B\diff(W)
$$ 
we get a long exact sequence of homotopy groups 
$$
\pi_1 (\Sigma ^H_G \diff_k (W)) \longrightarrow \pi_1 (\diff (W)) \longrightarrow 
\pi_1 (C_k (W; \Fr \times _G BH_1)) \longrightarrow  
\Sigma_G^H \Gamma_k (W) \longrightarrow
\Gamma (W) \longrightarrow 0.
$$
Using the maps in (3.1) we see  that the first map is split surjective, and hence we have a 
short exact sequence of discrete groups 
$$
0 \longrightarrow \pi_1 (C_k (W; \Fr \times _G BH_1)) \longrightarrow  
\Sigma_G^H \Gamma_k (W) \longrightarrow
\Gamma (W) \longrightarrow 0.
$$
For different $k$ these short exact sequences are compatible with the stabilisation maps. Consider the
Leray-Serre spectral sequence
$$
E^2_{p,q} = H_p( \Gamma (W), H_q (\pi_1 (C_k(W  ; \Fr \times _G BH_1) ))) 
\Longrightarrow H_{p+q} (   \Sigma_G^H \Gamma_k (W) ).
$$
The result 
now follows from Corollary 4.3 and an application of Lemmas 2.7 and 2.8. 
\end{proof}


\section {Variations and extensions}

Homology stability of configuration spaces has been studied extensively 
in recent years and the different versions and strengthenings of Theorem 2.2 
lead immediately to variants of our main theorems. We mention a few 
of them here.

\subsection{Changing coefficients}
The stability range of $k/2$ in our main theorems  
is a direct consequence of the stability range for the homology of 
the configuration spaces in Theorem 2.2.
The latter stability range can be improved 
when we change the  coefficients of the homology theory. 
For example, in \cite {CPalmer} it is shown that for $\dim W \geq 3$ 
the stability range can be improved to  $* \leq k-1$
if we take homology with coefficients in $\mathbb Z[ 1/2]$.
Hence, Theorems 1.1, 1.2, 1.3 and their mapping class group 
analogues Theorems 5.1, 5.2, 5.3 hold for such  manifolds 
 with this improved stability range. 
See \cite {CPalmer} and papers cited there for a discussion 
of field coefficients more generally.


\subsection{Boundary conditions and subgroups}
In order to define stabilisation maps, we considered only diffeomorphisms that fix point-wise a disk $\partial _0$ in the boundary of
$\partial W$. We may introduce any other compatible boundary conditions on 
the diffeomorphisms and analogues of our main theorems will hold. 

We may also consider any subgroups of $\diff (W)$ as long as the evaluation 
map $E_k$ restricts to a surjective map onto $C_k (W)$ and induces homotopy 
fibre sequences analogous to those in (3.2) and (3.3). 
Similarly, Theorem 3.4 allows us already to replace  
the group $\diff _G(Q_1)$ by any of its subgroups as long
as the restriction to the boundary defines a fibration over $G$.   


\subsection{Replacing the symmetric by the alternating group}
The three types of diffeomorphism groups that we consider  
all come equipped with a natural surjection to the symmetric group $\Sigma_k$
by considering the permutation of the $k$ punctures, deleted disks and 
copies of $Q\setminus \text{int } (D)$ respectively. This defines in particular 
short exact sequences of groups:
$$
P _G\diff (W\sharp_k Q) \longrightarrow \Sigma _G\diff (W\sharp _k Q) \longrightarrow \Sigma_k.
$$
We refer to these therefore as symmetric 
diffeomorphism groups as in the title. Similarly, we can define alternating 
diffeomorphism groups as  subgroups of the symmetric diffeomorphism groups fitting into the extension sequence
$$
P _G\diff (W\sharp_k Q) \longrightarrow A_G\diff (W\sharp_k Q) \longrightarrow A_k.
$$
Here $A_k \subset \Sigma _k$ denotes the alternating group. 
In \cite{Palmer} Palmer considered oriented configuration spaces with labels in a space $X$. These are defined as the orbit spaces
$$
C_k^+(W; X) := \widetilde C_k ( W; X)/ A_{k}
$$
of  ordered configuration spaces under the action of the alternating group 
$A_k$. Palmer
proves homology stability for these spaces: the stabilisation maps induce 
in homology 
isomorphisms for 
$* \leq (k-5)/3$ and surjections for $* \leq (k-2)/3$. 
This is enough for the analogues of Theorems 1.1, 1.2 and 1.3   
 with these stability conditions
to be proved for 
the alternating  diffeomorphism and mapping class groups
{\it as long as} $W$ is parallelisable and the frame  bundle $\Fr$ is trivial.

\subsection{Group completion}
Take $W=D$ to be the $d$-dimensional disk.
The union 
$$
\mathbb M :=\coprod_{k\geq 0} B\Sigma_G\diff(D \sharp _k Q)
$$  
forms a monoid induced by connected sum. 
Taking models of $B\Sigma_G \diff (D\sharp _kQ)$ of embedded manifolds 
in Euclidean space,  
one can construct an action of the standard little $d$-disk operad 
on $\mathbb M$. Thus its group completion $\Omega B \mathbb M$ 
is a $d$-fold loop space. As $d \geq 2$ the monoid $\mathbb M$ is 
therefore homotopy commutative. Define 
$$
\Sigma _G\diff (D \sharp_\infty Q) := \hocolim_k \Sigma _G\diff (D\sharp_k Q).
$$ 
By the strengthened group completion theorem, see \cite {RW2} \cite {MillerPalmer} and 
references therein, the 
fundamental group of the classifying space of the monoid $\mathbb M$
has 
perfect commutator subgroup and 
\begin{equation}
\tag{6.1}
\Omega B \mathbb M\simeq \mathbb Z \times B \Sigma_G \diff (D\sharp _\infty Q) ^+; 
\end{equation}
here $X^+$ denotes the Quillen plus construction of $X$ with 
respect to the maximal perfect subgroup of the fundamental group $\pi_1 (X)$. 
As this space is a 
homotopy commutative  H-space, its homology is a bi-commutative Hopf algebra. 
The structure of these 
has been studied in \cite {MilnorMoore}. 
In particular, if its homology with rational coefficients is of finite 
type then it is a graded polynomial algebra, 
that is a polynomial algebra on even dimensional generators tensor 
an exterior algebra on odd dimensional generators.
By the homology stability theorem, Theorem 1.3, we deduce that 
$$
H_*(B\Sigma_G\diff (D\sharp _k Q); \mathbb Q) = \mathbb Q[\,q_i \,; i \in I]
$$ 
is a polynomial algebra
in degrees 
less or equal to $k/2$ on generators $q_i$ indexed by some set $I$.

Similarly, we may  consider the monoid built from mapping class groups, 
$$
\mathbb M^\Gamma_Q:= \coprod_{k\geq 0} B \Sigma_G \Gamma (D\sharp_kQ).
$$
 Define 
$
\Sigma _G\Gamma (D \sharp_\infty Q) := \lim_k \Sigma_G \Gamma (D\sharp_k Q).
$ 
Then by the same reasoning as above, its 
commutator subgroup  is perfect and 
\begin{equation}
\tag{6.2}
\Omega B \mathbb M^\Gamma \simeq \mathbb Z \times B \Sigma_G \Gamma 
(D\sharp _\infty Q) ^+ 
\end{equation}
is a $d$-fold loop space. If its rational homology is of finite type then it is
a graded polynomial algebra.



	\addcontentsline{toc}{section}{References}
	\bibliography{RM_Paper_Biblio}
	\bibliographystyle{amsalpha}

\vskip .3in
Mathematical Institute
\newline
Oxford University
\newline
Andrew Wiles Building
\newline
Oxford OX2 6GG
\newline
UK

{\it tillmann@maths.ox.ac.uk}

\end{document}